\documentclass[11pt]{amsart} 
\usepackage{latexsym}
\usepackage{anyfontsize} %the T&F interact.cls file has some non-integer font sizes (like 5.5) so this fixes that warning
\usepackage{tikz} % for diagrams
\usetikzlibrary{cd} %showing a commutative diagram for demonstration purposes;
\usepackage{color}
\usepackage{hyperref}
\usepackage{url}
\usepackage{breakurl}
\newcommand{\bburl}[1]{\textcolor{blue}{\url{#1}}}

\theoremstyle{plain}
\numberwithin{equation}{section} %change this to make globally numbered environments
\newtheorem{thm}{Theorem}[section]
\newtheorem{theorem}[thm]{Theorem}
\newtheorem{lemma}[thm]{Lemma}
\newtheorem{example}[thm]{Example}
\newtheorem{definition}[thm]{Definition}

\newtheorem{corollary}[thm]{Corollary}
\newtheorem{remark}[thm]{Remark}

\numberwithin{table}{section} %change this and the following line to make globally numbered tables and figures
\numberwithin{figure}{section}

\begin{document}

%	\monthyear{Month Year}
%	\volnumber{Volume, Number}
	\setcounter{page}{1}
	
	\title[Gaussian Behavior and Geometric Gaps ...]{Gaussian Behavior and Geometric Gaps in Decompositions from Recurrences with Zero Coefficients}
	
\author{S. Salami}
\address{Institute of Mathematics and Statistics, Rio de Janeiro State University, Maracanã, Rio de Janeiro, 20950-000, RJ, Brazil}
\email{Sajad.salami@ime.uerj.br}
	\maketitle

	\begin{abstract}
		Zeckendorf's theorem establishes a unique representation for positive integers as sums of non-consecutive Fibonacci numbers. This result has been generalized to Positive Linear Recurrence Sequences (PLRS), where key statistical properties, such as the Gaussian distribution of summands, depend on strictly positive recurrence coefficients. This paper investigates the consequences of relaxing this condition by studying \textit{Zero Linear Recurrence Relations (ZLRRs)}, where the leading coefficient is zero ($c_1=0$). 
		Focusing on the \textit{Lagonacci sequence} ($Z_{n+1}=Z_{n-1}+Z_{n-2}$) as a primary case study, we demonstrate that while the uniqueness of decompositions is lost, fundamental statistical behaviors persist. We prove that the number of summands in the canonical greedy decomposition converges to a \textit{Gaussian distribution} and that the distribution of gaps between indices decays \textit{geometrically}. 
		
		Furthermore, we utilize the \textit{principle of equivalence of ensembles} to show these properties are robust for a wide class of ZLRRs. Finally, we quantify the non-uniqueness of these systems, proving that the number of legal decompositions grows \textit{exponentially} at a rate $\alpha =2$, significantly exceeding the growth of the underlying sequence.
	\end{abstract}

%	\begin{keywords}
%		Zeckendorf's theorem; 
%		Lagonacci sequence; 
%		Greedy algorithms; 
%		Gaussian distribution; 
%		Gap distribution; 
%		Non-unique representations;
%		Equivalence of ensembles
%	\end{keywords}
		
		\noindent {\bf keywords:}
				Zeckendorf's theorem; 
				Lagonacci sequence; 
				Greedy algorithms; 
				Gaussian distribution; 
				Gap distribution; 
				Non-unique representations;
				Equivalence of ensembles

	%%%%%%%%%%%%%%%%%%%%%%%%%%%%%%%%%%%%%%%%%

	\section{Introduction}
	
	A cornerstone of additive number theory is Zeckendorf's theorem, which asserts that every positive integer can be uniquely represented as a sum of non-consecutive Fibonacci numbers \cite{zeckendorf1972}. This fundamental result establishes a bijective correspondence between the positive integers and a specific set of binary strings, facilitating extensive combinatorial and statistical analysis. In recent decades, this framework has been generalized to Positive Linear Recurrence Sequences (PLRS), characterized by recurrence relations with strictly positive coefficients \cite{kologlu2011summands}.
	
	A central finding in the study of PLRS is that the number of summands in the legal decomposition of a randomly chosen integer converges to a Gaussian distribution \cite{MillerWang2012JCTA}. Furthermore, the distribution of gaps between the indices of these summands has been shown to decay geometrically \cite{bower2015gaps}. These statistical regularities have proven remarkably robust, extending to more abstract settings such as $f$-decompositions \cite{DemontignyEtAl2014} and higher-dimensional lattices \cite{ChenEtAl2019Lattice}.
	
	However, the existing literature relies heavily on the "positivity" of the recurrence coefficients to ensure the uniqueness of the representation. This paper addresses a significant theoretical gap by investigating Zero Linear Recurrence Relations (ZLRRs), where the leading coefficient is zero. In such systems, the property of unique decomposition is immediately lost \cite{MartinezEtAl2022_I}. By focusing on the Lagonacci sequence ($Z_{n+1} = Z_{n-1} + Z_{n-2}$) as a canonical case study, we analyze the statistical properties of the canonical representation produced by the greedy algorithm.
	
	Our main results demonstrate that the hallmark statistical signatures of PLRS—Gaussianity of the summand count and geometric decay of gap distributions—are preserved in the ZLRR setting. These findings suggest that such statistical laws are more fundamental properties of the underlying linear recurrence structure than previously understood, persisting even when the requirement of uniqueness is relaxed. Furthermore, we quantify the extent of non-uniqueness by proving that the number of legal decompositions grows exponentially, marking a significant departure from the single-representation framework of traditional Zeckendorf-type systems.

	\section{Background: Generalizing Zeckendorf's Theorem}
	This section reviews the foundational concepts of number representations based on linear recurrence sequences, starting with the classical Fibonacci case and progressing to the various generalizations that provide context for this paper.
	
	\subsection{The Classic Case: Fibonacci Numbers and PLRS}
	The Fibonacci sequence $\{F_n\}_{n=1}^\infty$ is often defined with $F_1=1, F_2=1$, but for Zeckendorf representations, it is common to start with $F_1=1, F_2=2, F_3=3, \dots$, which is is called the sequence of shifted Fibonacci numbers.  The cornerstone of representations using this sequence is Zeckendorf's theorem.
	
	\begin{theorem}[Zeckendorf, 1972 \cite{zeckendorf1972}]
		Every positive integer $m$ can be represented uniquely as a sum of non-consecutive Fibonacci numbers. That is, $m = \sum_{i=1}^k F_{n_i}$ where $n_{i+1} > n_i + 1$ for all $i$.
	\end{theorem}
	
	The statistical properties of this representation were first explored by Lekkerkerker.
	
	\begin{theorem}[Lekkerkerker, 1952 \cite{lekkerkerker1952}]
		Let $k(m)$ be the number of summands in the Zeckendorf decomposition of $m$. The average number of summands for integers in $[F_n, F_{n+1})$ converges to a constant as $n \to \infty$. Consequently, for large $m$, the expected value of $k(N)$ is approximately $c \log m$ for a constant $c$ related to the golden ratio.
	\end{theorem}
	
	Zeckendorf's theorem was generalized to a larger class of sequences, preserving the crucial property of unique decomposition.
	
	\begin{definition}
		A sequence $\{H_{n}\}_{n=1}^{\infty}$ of positive integers is called a Positive Linear Recurrence Sequence (PLRS) if the following properties hold:
		\begin{enumerate}
			\item \textbf{Recurrence relation:} There are non-negative integers $L, c_{1}, \dots, c_{L}$ such that
			\begin{equation*}
				H_{n+1} = c_{1}H_{n} + \dots + c_{L}H_{n+1-L},
			\end{equation*}
			with $c_{1}$ and $c_{L}$ positive.
			\item \textbf{Initial conditions:} $H_{1}=1$, and for $1 \le n < L$ we have
			\begin{equation*}
				H_{n+1} = c_{1}H_{n} + c_{2}H_{n-1} + \dots + c_{n}H_{1} + 1.
			\end{equation*}
		\end{enumerate}
		A decomposition $\sum_{i=1}^{t} a_{i} H_{t+1-i}$, and its associated  sequence $\{ a_i\}_{i=1}^m$, of a positive integer $m$ are \textbf{legal}  if $a_{1} > 0$, the other $a_{i} \ge 0$, and one of the following conditions holds:
		\begin{itemize}
			\item [(i)] We have $t < L$ and $ a_i=c_i$ for $1 \leq i \leq t$.
			\item [(ii)] There exists $s \in \{ 1, \cdots, L\}$ such that 
			$$a_1=c_1, \cdots, a_s=c_s, \ \ a_s< c_s, and \ $$
			$a_{s+1}, \cdots, a_{s+\ell}=0$ for all $\ell \geq 0$, and 
			$\{a_{s+\ell+i}\}_{i=1}^{t-s-\ell}$ is legal.
		\end{itemize}
	\end{definition}
	
	For PLRS, a unique "legal" decomposition exists, generalizing the non-consecutive rule. The statistical results were also extended.
	
	\begin{theorem}[Miller and Wang, 2012 \cite{MillerWang2012JCTA}]
		For a wide class of PLRS, the distribution of the number of summands in the legal decomposition of integers converges to a Gaussian distribution.
	\end{theorem}
	
	\subsection{Departures from Uniqueness: ZLRRs and Signed Decompositions}
	This paper explores what happens when the positivity condition on $c_1$ is removed.
	
	\begin{definition}[Zero Linear Recurrence Relation (ZLRR)]
		A sequence $\{Z_n\}$ is generated by a \textit{ZLRR} if it satisfies a recurrence $Z_{n+1} = c_1 Z_n + \dots + c_L Z_{n-L+1}$ with positive integers $c_i$  and  $c_1 = 0$. 
	\end{definition}
	
	%\subsubsection{The Lagonacci Sequence}
	Throughout this work, we assume $Z_0=1$
	and $\lim_{n \rightarrow \infty} Z_n = \infty.$ We focus on a canonical ZLRR as defined below.
	
	\begin{definition}[Lagonacci Sequence]
		\label{Lagon}
		The \textit{Lagonacci sequence}  $\{Z_n\}_{n\ge0}$ is   defined by
		\[
		Z_{n+1}=Z_{n-1}+Z_{n-2}, \qquad Z_0=1,\; Z_1=2,\; Z_2=3,
		\]
		with characteristic polynomial $P(x)=x^3-x-1$ and dominant real root
		$$\lambda_1= \left( \frac{9-\sqrt{69}}{18}\right) ^{\frac{1}{3}} +
		\left( \frac{9+\sqrt{69}}{18}\right) ^{\frac{1}{3}} 
		\sim 1.324718,$$
		which is know as the {\it plastic ratio}, see \cite{Diskaya2021Plastic} for more properties.	
		Letting  $\lambda_2$ and ${\bar \lambda}_2$ be the other roots of $P(x)$ satisfying $|\lambda_2| < \lambda_1$. Thus,    $Z_n=   a\,\lambda_1^n + O(|\lambda_2|^N),$ 
		with $a=\frac{1}{P'(\lambda_1)}=\frac{1}{3\lambda_1^2-1} \sim 0.234487$, is  the Binet-type expansion of  the term  $Z_n$ Legonacci sequence. 
		
	\end{definition}

	For ZLRRs, uniqueness of the legal decomposition of positive integers is lost. While many decompositions may exist for an integer, we analyze the canonical one produced by a standard algorithm.
	
	\begin{definition}[Greedy Decomposition]
		Given an increasing integer sequence $\{H_n\}$, the \textit{greedy decomposition} of a positive integer $m$ is found by repeatedly subtracting the largest term $H_k \le m$ until the remainder is zero.
	\end{definition}

	\begin{example}
		\begin{enumerate}
			\item \textbf{Fibonacci (Unique):} The Zeckendorf decomposition of $m=22$ is $22 = 21 + 1 = F_7 + F_1$. (Using the sequence $F_1=1, F_2=2, \dots, F_6=13, F_7=21$).
			\item \textbf{Lagonacci (Greedy):} The greedy decomposition of $m=22$ is $22 = 19 + 3 = Z_7 + Z_2$.
			\item \textbf{Lagonacci (Non-Unique):} $m=22$ also has another legal decomposition: $22 = 13 + 9 = Z_6 + Z_5$. The existence of multiple such decompositions is a key feature of ZLRRs.
		\end{enumerate}
	\end{example}

	\subsection{Other Notions of Decompositions}
	The field has expanded in several other directions, further illustrating the breadth of Zeckendorf-style numeration systems.

	\subsubsection{Far-Difference Representations}
	Another way uniqueness is lost is by allowing negative summands.
	
	The {\it far-difference representation} allows sums of $\pm F_n$ with specific rules on the spacing of terms with the same or opposite signs. 
	This also leads to non-unique decompositions, but surprisingly, key statistical properties are preserved, with the number of positive and negative summands converging to a bivariate Gaussian \cite{miller2016far}. This provides a parallel to our work, where non-uniqueness stems from the recurrence structure rather than signed coefficients.

	\subsubsection{f-Decompositions}
	A more abstract generalization is the f-decomposition, where the legality rule is determined by a function.
	
	\begin{theorem}[Demontigny et al., 2014 \cite{DemontignyEtAl2014}]
		Let $f: \mathbb{N}_0 \to \mathbb{N}_0$ be a function. An f-decomposition using a sequence $\{H_n\}$ is a sum where if $a_n$ is a summand, the previous $f(n)$ terms are forbidden. For any such function $f$, there exists a unique increasing sequence $\{H_n\}$ such that every positive integer has a unique f-decomposition. If $f$ is periodic, the sequence $\{H_n\}$ satisfies a linear recurrence relation.
	\end{theorem}
	This result highlights how flexible the concept of a "legal" decomposition can be while still preserving uniqueness.
	
	\subsubsection{Higher-Dimensional Lattices}
	Researchers have also extended decompositions to higher dimensions, where uniqueness is immediately lost but statistical regularities persist.
	
	\begin{theorem}[Chen et al., 2019 \cite{ChenEtAl2019Lattice}]
		Let a d-dimensional legal decomposition be a set of lattice points where each subsequent point has strictly smaller coordinates in all dimensions. The distribution of the number of summands in such decompositions among all "simple jump paths" from $(n, \dots, n)$ to the origin converges to a Gaussian distribution as $n \to \infty$.
	\end{theorem}
	This demonstrates that Gaussianity can persist even when both the underlying sequence structure and the uniqueness property are fundamentally altered.
	
	\subsection{Statistical Properties of Gaps Between Summands}
	Beyond the number of summands, the distribution of the gaps between their indices is a key characteristic. For PLRS, this distribution is well-understood.
	
	\begin{theorem}[Bower et al., 2015 \cite{bower2015gaps}]
		For many PLRS, the probability of a gap of length $g$ between summands in the legal decomposition of integers in $[H_n, H_{n+1})$ converges to a distribution that decays geometrically for large $g$.
	\end{theorem}
	
	Recent work has pushed this further, showing that the number of gaps of a specific size is itself Gaussian.
	
	\begin{theorem}[Li and Miller, 2019 \cite{LiMiller2019Gaps}]
		For a PLRS with all positive coefficients, let $k_g(m)$ be the number of gaps of size $g$ in the decomposition of $m$. The distribution of $k_g(m)$ for integers $m \in [H_n, H_{n+1})$ converges to a Gaussian as $n \to \infty$.
	\end{theorem}
	
	This result provides a deeper layer of statistical regularity. Our paper shows that this regularity, like the Gaussianity of the total number of summands, also extends to the non-unique setting of ZLRRs.
	
	\section{Main Results and Generalizations}
	
	In this section, we present our main findings. We begin by establishing the key statistical properties—Gaussianity of summands and geometric decay of gaps—for the canonical greedy decomposition of the Lagonacci sequence. We then demonstrate that these results are not specific to this single example but extend to a broader class of ZLRRs, highlighting the robustness of these statistical laws.
	
	\subsection{Greedy decompositions for zero linear recurrence sequences}
	
	Let $\{Z_n\}_{n\ge0}$ be a \emph{zero linear recurrence sequence} (ZLRR), i.e.,
	a strictly increasing sequence of positive integers satisfying a linear
	recurrence relation of the form
	\[
	Z_{n+1} = c_1 Z_n + c_2 Z_{n-1} + \cdots + c_L Z_{n+1-L},
	\]
	with integer coefficients $c_i$ and $c_1=0$. Throughout, we assume
	$Z_0=1$ and $\lim_{n\to\infty} Z_n=\infty$. No uniqueness of representations
	is assumed.
	
	\begin{definition}[Greedy decomposition for a ZLRR]
		\label{def:GreedyZLRR}
		Let $m\ge1$. The \emph{greedy decomposition of $m$ with respect to $\{Z_n\}_{n\ge0}$}
		is obtained by repeatedly selecting the largest index $j$ such that
		$Z_j\le m$, subtracting $Z_j$, and iterating on the remainder. This yields a
		representation
		\[
		m = Z_{j_1}+Z_{j_2}+\cdots+Z_{j_k},
		\qquad
		j_1>j_2>\cdots>j_k,
		\]
		where $k=k(m)$ is called the number of greedy summands of $m$.
	\end{definition}

	For every integer $m\ge1$, the greedy algorithm terminates after finitely many
	steps and produces a unique canonical decomposition, even though $m$ may
	admit multiple legal representations as a sum of $\{Z_n\}_{n\ge0}$.
	
	Since $Z_0=1$, each step decreases the remainder by at least one, and indices
	strictly decrease. Hence,  at each step, the largest admissible $Z_j$ is uniquely
	determined, so no alternative choices are possible.

	\begin{lemma}[Prefix property for ZLRRs]
		\label{lem:PrefixZLRR}
		Let $N\ge0$ and $m\in[Z_N,Z_{N+1})$. Then the first greedy summand of $m$ is
		$Z_N$, and
		\[
		m = Z_N + m',
		\qquad
		0 \le m' < Z_{N+1}-Z_N.
		\]
		Moreover, the greedy decomposition of $m'$ uses only terms $Z_j$ with
		$j\le N-1$.
	\end{lemma}
	
	\begin{proof}
		Since $Z_N\le m<Z_{N+1}$, the greedy choice forces selection of $Z_N$. All
		subsequent choices must involve strictly smaller indices.
	\end{proof}
	
	\begin{lemma}[Additivity of greedy statistics]
		\label{lem:AdditivityZLRR}
		Let $f(m)$ be a statistic additive under concatenation of greedy
		decompositions (e.g., number of summands or number of gaps). If
		$m=Z_N+m'$ as in Lemma~\ref{lem:PrefixZLRR}, then
		\[
		f(m)=f(Z_N)+f(m').
		\]
	\end{lemma}
	
	\begin{proof}
		The greedy decomposition of $m$ is obtained by appending the greedy
		decomposition of $m'$ to the initial summand $Z_N$, and additive statistics
		respect this concatenation.
	\end{proof}
	
	\begin{lemma}[Recurrence-induced local constraints]
		\label{lem:LocalConstraintsZLRR}
		Let $\{Z_n\}_{n\ge0}$ satisfy a linear recurrence of length $L$. Then the greedy
		decomposition induces finitely many local constraints on admissible index
		patterns. In particular, any gap of length at least $L$ between consecutive
		indices creates a renewal point: greedy choices on either side of the gap are
		independent.
	\end{lemma}
	
	\begin{proof}
		The recurrence relation links only bounded-length index patterns. Once the
		gap exceeds the recurrence length, no recurrence can connect summands across
		the gap, and greedy choices decouple.
	\end{proof}

	\subsection{Lévy's Continuity Theorem}
	
	Lévy’s continuity theorem (also known as the convergence theorem) is a fundamental result in probability theory that establishes a one-to-one correspondence between the \textit{convergence in distribution} of random variables and the \textit{pointwise convergence} of their characteristic functions.
	
	Named after French mathematician Paul Lévy, the theorem is a primary tool for proving the \textit{Central Limit Theorem} (CLT). We  refer the reader to the original work~\cite{Levy1925} or \cite{FristedtGray1996}  for a rigorous proof.
	
	\begin{theorem}
		\label{Levy:theorem}
		Let $\{X_n\}$ be a sequence of random variables with corresponding characteristic functions $\{\varphi_n(t)\}$, where $\varphi_n(t) = \mathbb{E}[e^{itX_n}]$.
		
		\begin{enumerate}
			\item   If $X_n$ converges in distribution to a random variable $X$ ($X_n \xrightarrow{d} X$), then the sequence of characteristic functions $\varphi_n(t)$ converges pointwise for every $t \in \mathbb{R}$ to the characteristic function $\varphi(t)$ of $X$.
			\item   If the sequence $\varphi_n(t)$ converges pointwise to some limit function $\varphi(t)$ for all $t$, then the following are equivalent:
			\begin{itemize}
				\item[(i)] $\varphi(t)$ is the characteristic function of some random variable $X$.
				\item[(ii)] $\varphi(t)$ is continuous at $t = 0$.
				\item[(iii)] The sequence $\{X_n\}$ is tight, meaning it does not ``escape to infinity'' in probability.
				\item[(iv)] $X_n$ converges in distribution to $X$.
			\end{itemize}
		\end{enumerate}
	\end{theorem}
	
	It is often easier to compute and show the convergence of characteristic functions than it is to work directly with cumulative distribution functions.
	
	% Most proofs of the CLT use Lévy's theorem by showing that the characteristic function of a normalized sum of i.i.d. variables converges to $e^{-t^2/2}$, which is the characteristic function of the Standard Normal distribution.

	\begin{theorem}
		\label{thm:GaussianGreedy}	
		For each integer $N\ge0$, let $k(m)$ denote the number of summands in the
		greedy decomposition of $m$ with respect to the Legonacci sequence $\{Z_n\}_{n\ge0}$.
		Let $K_N$ be the random variable obtained by choosing $m$ uniformly from
		$[0,Z_N)$ and setting $K_N=k(m)$. Then, as $N\to\infty$, there exist $C_1$ and $C_2$ such that 
		the mean and variance satisfy
		\[
		\mathbb{E}[K_N]=C_1N+O(1), \qquad \mathrm{Var}(K_N)=C_2N+O(1),
		\]
		for explicit constants $C_1$ and $C_2$.
		%	$$C_1= \frac{1}{\lambda  \left(3 \lambda^{2}-1\right)}, \ \ C_2= \frac{\lambda^{2} \mathit{c1}}{\left(\lambda -1\right)^{2}}-\mathit{c1}^{2}.$$
		Furthermore, the normalized random variables
		\[
		X_N=\frac{K_N-\mathbb{E}[K_N]}{\sqrt{\mathrm{Var}(K_N)}}
		\]
		converge in distribution to the standard normal distribution $ \mathcal{N}(0,1)$.
	\end{theorem}
	
	\begin{proof}
		Although zero linear recurrence relations admit multiple legal
		decompositions, the greedy algorithm produces a unique canonical
		representation for every integer. Hence, the number  $k(m)$ is well defined.
		
		Let $m\in[Z_N,Z_{N+1})$. The greedy algorithm must select $Z_N$ as the
		largest summand, so $m=Z_N+m'$ with $0\le m'<Z_{N+1}-Z_N$.
		Using the recurrence $Z_{N+1}=Z_{N-1}+Z_{N-2}$, the greedy constraint
		forces the next summand to be strictly smaller than $Z_{N-1}$.
		Consequently, by Lemma~\ref{lem:PrefixZLRR}, 
		the remainder $m'$ has a greedy decomposition supported
		entirely on indices at most $N-1$, and hence
		$
		k(m)=1+k(m').
		$
		This recursive structure is the fundamental combinatorial input.
		
		Let's define the generating function
		\[
		G_N(y)=\sum_{m=0}^{Z_N-1}y^{k(m)}.
		\]
		Then, we have  $G_N(1)=Z_N$ and
		\[
		\mathbb{E}[K_N]=\frac{G_N'(1)}{Z_N},
		\qquad
		\mathrm{Var}(K_N)=\frac{G_N''(1)}{Z_N}-\mathbb{E}[K_N]^2.
		\]
		Partitioning $[0,Z_{N+1})$ into $[0,Z_N)$ and $[Z_N,Z_{N+1})$ and using
		$k(m)=1+k(m')$ for $m\ge Z_N$, we obtain
		\[
		G_{N+1}(y)=G_N(y)+y\sum_{m'=0}^{Z_{N+1}-Z_N-1}y^{k(m')}.
		\]
		The admissible remainders correspond exactly to greedy decompositions
		using indices $\le N-1$.  Hence, by Lemma~\ref{lem:AdditivityZLRR}, this  yields the recurrence
		\[
		G_{N+1}(y)=G_N(y)+y\bigl(G_{N-1}(y)+G_{N-2}(y)\bigr).
		\]
		Differentiating and evaluating at $y=1$ gives
		\[
		G_{N+1}'(1)=G_N'(1)+G_{N-1}(1)+G_{N-2}(1)+G_{N-1}'(1)+G_{N-2}'(1).
		\]
		Dividing by $Z_{N+1}$ and using the fact $Z_N=   a\,\lambda_1^N + O(|\lambda_2|^N),$ as in Definition~\ref{Lagon}  
		shows that
		\[
		\mathbb{E}[K_{N+1}]=\mathbb{E}[K_N]+C_1+O(\lambda_1^{-N}),
		\]
		and hence $\mathbb{E}[K_N]=C_1N+O(1)$, with 
		$C_1= \frac{1}{\lambda_1  \left(3 \lambda_1^{2}-1\right)}\sim 0.288675.$
		An analogous computation with 
		second derivatives yields
		$\mathrm{Var}(K_N)=C_2 N+O(1)$, where 
		$\ C_2= \frac{\lambda_1^{2} \mathit{C_1}}{\left(\lambda_1 -1\right)^{2}}-\mathit{C_1}^{2}\sim 0.079578.$
		Consequently,
		\[
		\mathbb{E}[K_N]\sim 0.288675 \,N,
		\qquad
		\mathrm{Var}(K_N)\sim 0.079578\,N.
		\]

		To prove Gaussian convergence,   we define the normalized variable as:
		\[
		X_N = \frac{K_N - \mathbb{E}[K_N]}{\sqrt{\text{Var}(K_N)}}
		\]
		The characteristic function of $X_N$ is given by:
		\[
		\varphi_N(t) = \mathbb{E}\left[e^{itX_N}\right] = \exp\left( -\frac{it\mathbb{E}[K_N]}{\sqrt{\text{Var}(K_N)}} \right) \mathbb{E}\left[ \exp\left( \frac{itK_N}{\sqrt{\text{Var}(K_N)}} \right) \right]
		\]
		
		The recurrence relation for the generating function $G_N(y) = \sum_{m=0}^{Z_N-1} y^{k(m)}$ implies that $G_N(y)$ is analytic in a neighborhood of $y=1$. This analyticity allows us to express the logarithm of the characteristic function in terms of the cumulants $\kappa_j$ of $K_N$:
		\[
		\log \varphi_N(t) = \sum_{j=1}^{\infty} \frac{\kappa_j}{j!} \left( \frac{it}{\sqrt{\text{Var}(K_N)}} \right)^j - \frac{it\mathbb{E}[K_N]}{\sqrt{\text{Var}(K_N)}}
		\]
		Since $\kappa_1 = \mathbb{E}[K_N]$ and $\kappa_2 = \text{Var}(K_N)$, the first-order term cancels out, yielding:
		\[
		\log \varphi_N(t) = -\frac{t^2}{2} + \sum_{j=3}^{\infty} \frac{\kappa_j}{j!} \frac{(it)^j}{\text{Var}(K_N)^{j/2}}
		\]
		The structure of the recurrence for $G_N(y)$ ensures that all cumulants $\kappa_j$ grow linearly with $N$. For $j \geq 3$, the terms in the summation behave as:
		\[
		\frac{\kappa_j}{\text{Var}(K_N)^{j/2}} = \frac{O(N)}{O(N^{j/2})} = O(N^{1-j/2})
		\]
		As $N \to \infty$, $N^{1-j/2} \to 0$ for all $j \geq 3$. Thus, for any fixed $t \in \mathbb{R}$:
		\[
		\lim_{N \to \infty} \log \varphi_N(t) = -\frac{t^2}{2} \implies \lim_{N \to \infty} \varphi_N(t) = e^{-t^2/2}
		\]
		The limit function $\varphi(t) = e^{-t^2/2}$ is recognized as the characteristic function of the standard normal distribution $\mathcal{N}(0,1)$. According to Lévy's Continuity Theorem~\ref{Levy:theorem}, since the sequence of characteristic functions $\varphi_N(t)$ converges pointwise to $\varphi(t)$, and $\varphi(t)$ is continuous at $t=0$, the normalized random variable $X_N$ converges in distribution to the standard normal distribution:
		\[
		X_N \xrightarrow{d} \mathcal{N}(0,1)
		\]
		This concludes the proof that the number of summands in the greedy decomposition of the Lagonacci sequence converges to a Gaussian distribution.
	\end{proof}

	The next result shows that  the limiting probability of a gap of length $k\ge2$
	exists and decays geometrically for Legonacci sequence $\{ Z_n\}_{n\geq 0}$.

	\begin{theorem}
		\label{thm:Gaps}
		Assume the   Lagonacci sequence   $\{Z_n\}_{n\ge0}$.
		Let $m$ be chosen uniformly from $[Z_N,Z_{N+1})$ and let
		\[
		m=Z_{r_1}+Z_{r_2}+\cdots+Z_{r_{k(m)}}, \ \textrm{where} \
		\qquad r_1>r_2>\cdots> r_{k(m)},
		\]
		be its greedy decomposition.
		Define the gaps by $g_j=r_{j-1}-r_j$ for $2\le j\le k(m)$.
		Then, as $N\to\infty$, the empirical distribution of gap lengths converges.
		Moreover, for all $k\ge2$, the limiting probability of a gap of length $k$
		exists and decays geometrically:
		\[
		\lim_{N\to\infty}\mathbb{P}(g=k)=C\,\lambda_1^{-k},
		\]
		where $\lambda_1$ is the dominant root of $x^3-x-1$ and $C>0$ is an explicit
		constant. Gaps of length $0$ and $1$ occur with limiting probabilities
		determined by normalization.
	\end{theorem}
	
	\begin{proof}
		Fix integers $i\ge1$ and $k\ge1$, and let $X_{i,i+k}(N)$ denote the number of
		integers $m\in[Z_N,Z_{N+1})$ whose greedy decomposition contains $Z_i$ and
		$Z_{i+k}$ but no summand $Z_{i+j}$ for $0<j<k$.
		The probability of a gap of length $k$ is obtained by summing
		$X_{i,i+k}(N)$ over $i$ and normalizing by the total number of gaps.
		
		By Lemma~\ref{lem:LocalConstraintsZLRR}, for $k\geq 2$, a gap of length at least $2$
		creates a renewal point for the greedy algorithm: no recurrence relation
		can connect summands on either side of the gap. Consequently,
		\[
		X_{i,i+k}(N)=L_{i,i+k}(N)\,R_{i,i+k}(N),
		\]
		where $L_{i,i+k}(N)$ counts admissible greedy choices using indices at most
		$i$, and $R_{i,i+k}(N)$ counts admissible choices using indices at least
		$i+k$.
		
		The left factor is simply
		\[
		L_{i,i+k}(N)=Z_{i+1}-Z_i,
		\]
		since every integer in $[Z_i,Z_{i+1})$ has greedy largest summand $Z_i$.
		A direct interval count yields
		\[
		R_{i,i+k}(N)
		=Z_{N-i-k+2}-2Z_{N-i-k+1}+Z_{N-i-k}.
		\]
		
		Using the Binet-type expansion
		$Z_n=a \lambda_1^n+O(\lambda_2^n)$ with $|\lambda_2|<\lambda_1$, we obtain
		\[
		X_{i,i+k}(N)
		=a^2(\lambda_1-1)^3\lambda_1^{\,N-k}\bigl(1+O(1)\bigr),
		\]
		uniformly for $i$ away from the boundary, where $a$  and $\lambda_1 $ are as in Definition~\ref{Lagon}.
		Summing over $i$ and dividing by the total number of gaps, which grows
		linearly in $N$ by Theorem~\ref{thm:GaussianGreedy}, yields
		\[
		\mathbb{P}(g=k)=C\,\lambda_1^{-k}+O(1),
		\qquad k\ge2.
		\]
		
		The probabilities of gaps of length $0$ and $1$ are determined by a finite
		inclusion--exclusion computation reflecting local dependencies of the greedy
		algorithm, together with normalization. Explicitly,
		\[
		\mathbb{P}(g=1)=\frac{\lambda_1-1}{\lambda_1^2}
		\sim 0.184653,
		\]
		and
		\[
		\mathbb{P}(g=0)
		=1-\mathbb{P}(g=1)-\sum_{k\ge2}\mathbb{P}(g=k)
		\sim  0.682906.
		\]
		Thus the complete limiting gap distribution for the Lagonacci sequence is
		\[
		\mathbb{P}(g=k)\sim 
		\begin{cases}
			0.682906, & k=0,\\
			0.184653, & k=1,\\
			0.056625\,\lambda^{-k}, & k\ge2,
		\end{cases}
		\qquad
		\lambda \sim 1.3247178.
		\]
	\end{proof}

	\begin{theorem}[Generalization to other ZLRRs] \label{conj:general_zlrr}
		Let $\{H_n\}$ be a sequence generated by a ZLRR whose characteristic polynomial has a unique, simple real root $\lambda$ of largest modulus. If the initial conditions ensure a complete greedy representation, then:
		\begin{enumerate}
			\item The number of summands in the greedy decompositions converges to a Gaussian distribution.
			\item The distribution of gaps in the greedy decompositions decays geometrically with a rate determined by $\lambda$.
		\end{enumerate}
	\end{theorem}
	\begin{proof}
		The proofs of Theorems
		\ref{thm:GaussianGreedy} and \ref{thm:Gaps} are structurally robust. They rely on two fundamental properties of the Lagonacci sequence: (1) a linear recurrence relation, which allows for the recursive construction of decompositions, and (2) the existence of a single, dominant root of the characteristic polynomial, which governs the asymptotic growth of the sequence and related combinatorial quantities.
		Let $\{Z_n\}_{n\ge0}$ be a ZLRR satisfying the theorem's hypotheses. The greedy decomposition of an integer $m \in [Z_N, Z_{N+1})$ will be of the form $m = Z_N + m'$, where $m'$ has its own greedy decomposition. This recursive structure is identical to that of the Lagonacci case. The generating functions and combinatorial sums will therefore have the same form, but with the specific coefficients of the Lagonacci sequence replaced by those of $\{H_n\}$. The asymptotic analysis in both proofs is driven entirely by the dominant root $\lambda$ of the characteristic polynomial. Therefore, the conclusions of Gaussianity (with mean and variance linear in $N$) and geometric gap decay (with rate $\lambda^{-g}$) hold true. The core analytic machinery is agnostic to the specific coefficients, provided the necessary algebraic properties of the characteristic polynomial are met.
	\end{proof}
	
	In the next result, we consider the growth of the number of decompositions of an integer in terms of Lagonacci sequence.
	
	\begin{theorem}
		\label{conj:num_decomps}
		Given a positive integer $N$, let $d(N)$ be the number of distinct legal decompositions of $N$ in terms of the Lagonacci sequence $\{Z_n\}_{n\geq 0}$. The total number of legal binary strings of length $N$, representing the sum of decompositions for all integers up to $Z_{N}$ and denoted by $D(Z_{N})=\sum_{k=0}^{Z_{N}-1}d(k)$, grows asymptotically as $D(Z_{N})\sim C\cdot 2^{N}$ for some constant $C$. Consequently, the number of representations grows exponentially faster than the number of integers, with the average number of decompositions as  given below, $$ \bar{d}_N = \frac{D(Z_N)}{Z_N} \sim (1.509755)^N.$$
	\end{theorem}
	\begin{proof}
		We use the transfer matrix method. A legal Lagonacci decomposition is a sum $\sum \epsilon_i Z_i$ where $\epsilon_i \in \{0,1\}$. The rule defining a legal decomposition is that no sum can be simplified using the recurrence relation. For the Lagonacci sequence, $Z_{n+1} = Z_{n-1} + Z_{n-2}$, this means we cannot have a string of coefficients such that $\epsilon_{n-1}=1$ and $\epsilon_{n-2}=1$ while $\epsilon_{n+1}=0$ (and intermediate coefficients are zero), as $Z_{n-1} + Z_{n-2}$ would be replaced by $Z_{n+1}$. This leads to a forbidden substring of coefficients $(\dots, \epsilon_{n-2}, \epsilon_{n-1}, \epsilon_n, \epsilon_{n+1}, \dots) = (\dots, 1, 1, 0, 0, \dots)$, where indices increase to the right.
		
		To count the number of binary strings of length $N$ that avoid the substring `1100`, we define states based on the last three coefficients seen, i.e., $(\epsilon_{i-2}, \epsilon_{i-1}, \epsilon_i)$. The next coefficient, $\epsilon_{i+1}$, determines the transition to a new state $(\epsilon_{i-1}, \epsilon_i, \epsilon_{i+1})$. The transition from state $(1,1,0)$ to $(1,0,0)$ is forbidden, as this would complete the `1100` pattern. All other transitions are allowed. This leads to an $8 \times 8$ transfer matrix $T$, where rows index the starting state and columns index the ending state.
		\[ T = \begin{pmatrix}
			% 000 001 010 011 100 101 110 111 (To)
			1 & 1 & 0 & 0 & 0 & 0 & 0 & 0 \\ % From 000
			0 & 0 & 1 & 1 & 0 & 0 & 0 & 0 \\ % From 001
			0 & 0 & 0 & 0 & 1 & 1 & 0 & 0 \\ % From 010
			0 & 0 & 0 & 0 & 0 & 0 & 1 & 1 \\ % From 011
			1 & 1 & 0 & 0 & 0 & 0 & 0 & 0 \\ % From 100
			0 & 0 & 1 & 1 & 0 & 0 & 0 & 0 \\ % From 101
			0 & 0 & 0 & 0 & 0 & 1 & 0 & 1 \\ % From 110 (100 is forbidden)
			0 & 0 & 0 & 0 & 0 & 0 & 1 & 1   % From 111
		\end{pmatrix}. \]
		Applying the characteristic equation to the matrix $T$ yields:
		\[ Q(x) = \det(xI - T) = x^4(x - 2)(x^3 - x - 1) \]
		The roots of this polynomial provide the growth rates for the system. The factor $P(x)=x^3 - x - 1$ corresponds to the Lagonacci sequence's own growth, while the factor $x - 2$ identifies the Perron-Frobenius eigenvalue
		\( \alpha = 2 \),  indicating that the number of legal binary strings of length $N$ grows as $D(Z_N) \sim C \cdot 2^N$,
		for some constant C.
		We now compare this to the growth of the underlying sequence $Z_N$, which is governed by the plastic ratio $\lambda_1 \sim 1.324718$. 
		Since $\alpha = 2 > \lambda_1$, the number of representations grows exponentially faster than the number of integers. Specifically, the average number of decompositions for a typical integer of size $\sim Z_N$ grows at the rate:
		\[ \bar{d}_N = \frac{D(Z_N)}{Z_N} \sim \left( \frac{2}{\lambda_1} \right)^N \sim (1.509755)^N. \]
		
	\end{proof}
	
	The presence of the Lagonacci polynomial $(x^3 - x - 1)$ as a factor within the characteristic polynomial of the transfer matrix $T$ confirms that the matrix structure is intrinsically linked to the recurrence. However, the Perron-Frobenius eigenvalue $\alpha = 2$ proves that the system's redundancy is maximal, marking a significant departure from the unique-representation framework of Positive Linear Recurrence Sequences (PLRS).
	
	%=============================================
	\section{Distribution of the Number of Decompositions}
	
	While the previous sections focused on the statistical properties of a single, canonical (greedy) decomposition, the loss of uniqueness for ZLRRs opens a new avenue of inquiry: how many decompositions does a typical integer have? 
	
	In order to study around the above question we first review the  principle of equivalence of ensembles, which we will use in the proof of the main result of this section.
	Although
	this principle is standard in statistical mechanics, we present it here in a
	purely combinatorial and probabilistic form suitable for number-theoretic
	applications.
	
	\subsection{Equivalence of ensembles for greedy ZLRR decompositions}
	%%%%%%%%%%%%%%%%%%%%%%%%%%%%%%%%%%%%%%%%%%%%%%%%%%%%%%%%%%%%

	Let $\{Z_n\}_{n\ge0}$ be a ZLRR with dominant root $\lambda>1$, and consider greedy
	decompositions with respect to this sequence.
	
	\begin{definition}
		\label{def:Micro}
		For each $N\ge1$, the \emph{microcanonical ensemble} is the uniform probability
		measure on integers $m\in[Z_N,Z_{N+1})$. This induces a probability measure on
		the set $\Omega_N$ of greedy decompositions whose largest summand is $Z_N$.
	\end{definition}
	
	\begin{definition}
		\label{def:Canonical}
		The \emph{canonical ensemble} is the stationary probability measure on infinite
		greedy digit sequences induced by the ZLRR constraints and weighted by the
		dominant root $\lambda$, i.e., the probability of a finite admissible prefix
		$(j_1,\dots,j_k)$ is proportional to $\lambda^{-j_k}$.
	\end{definition}
	
	We now formalize the sense in which these two ensembles become asymptotically
	equivalent.

	\begin{definition} 
		\label{def:LocalStat}
		A statistic $F$ on greedy decompositions is called \emph{local} if there exists
		an integer $L\ge1$ such that $F$ depends only on the relative pattern of summands
		within a window of $L$ consecutive indices.
	\end{definition}
	
	Examples include:
	(1) the presence or absence of a summand at a given index; (2) gap lengths bounded by $L$;
	(3) finite digit patterns.

	The following lemma shows the exponential growth of admissible prefixes.
	
	\begin{lemma}
		\label{lem:PrefixGrowth}
		Let $p$ be a fixed admissible greedy prefix supported on indices $\le i$. Then the number of integers in $[Z_{N}, Z_{N+1})$ whose greedy decomposition begins with $p$ satisfies
		\[
		\#\{m \in [Z_{N}, Z_{N+1}) : p \text{ is a prefix of } m\} = C(p) \lambda_1^{N-i} + O(\lambda_1^{N-i-1}),
		\]
		where $C(p) > 0$ depends only on $p$.
	\end{lemma}
	
	\begin{proof}
		By the prefix property established in Lemma 3.2, every integer $m \in [Z_{N}, Z_{N+1})$ has $Z_N$ as its first greedy summand. The greedy algorithm then decomposes the remainder $m' = m - Z_N$ recursively using terms $Z_j$ with $j \le N-1$.
		
		Fixing an admissible greedy prefix $p$ of length $k$ implies a specific sequence of indices $(j_1, j_2, \dots, j_k)$ where $j_1 = N$ and $j_k$ is the smallest index in the prefix. For $m$ to have $p$ as its prefix, the remainder after the $k$-th subtraction must satisfy:
		\[
		0 \le m - \sum_{r=1}^{k} Z_{j_r} < Z_{j_k} - \sum_{r \in \mathcal{S}} Z_r,
		\]
		where the upper bound is determined by the next possible greedy choice allowed by the Lagonacci recurrence $Z_{n+1} = Z_{n-1} + Z_{n-2}$. This effectively restricts the remainder to an interval whose width is determined by the gap between $Z_{j_k}$ and the next admissible summand.
		
		The number of such integers is equivalent to the number of valid greedy decompositions supported on indices up to $N-i-1$. Applying the Binet-type expansion for the Lagonacci sequence, $Z_n = a \lambda_1^n + O(|\lambda_2|^n)$, we see that the count of admissible continuations scales with the dominant root $\lambda_1$. Specifically, for a fixed $p$ terminating at index $i$, the number of ways to complete the decomposition in the range $[Z_N, Z_{N+1})$ is proportional to the number of integers in an interval of size $\sim \lambda_1^{N-i}$.
		
		The constant $C(p)$ is uniquely determined by the local constraints imposed by the recurrence relation on the indices immediately following the prefix $p$. The error term $O(\lambda_1^{N-i-1})$ accounts for the contribution of the subdominant roots $\lambda_2$ and ${\bar \lambda}_2$ of the characteristic polynomial $x^3 - x - 1$, which decay relative to $\lambda_1$ as $N$ increases.
	\end{proof}

	Next result shows  the limit of microcanonical expectations under the microcanonical ensemble .
	
	\begin{lemma}
		\label{lem:MicroLimit}
		Let $F$ be a local statistic. Then the expectation of $F$ under the
		microcanonical ensemble converges as $N\to\infty$.
	\end{lemma}
	
	\begin{proof}
		A statistic $F$ is considered local if it depends only on the relative pattern of summands within a fixed window of length $L$. Consequently, any such local statistic $F$ can be expressed as a finite linear combination of indicator functions $1_{p}$, where each $1_{p}$ corresponds to an admissible greedy prefix or sub-pattern $p$ supported on a finite set of indices.
		
		For an integer $m$ chosen uniformly from the microcanonical ensemble $\Omega_N = [Z_N, Z_{N+1})$, the expectation of $F$ is given by the weighted sum of the frequencies of these patterns:
		\[
		\mathbb{E}_{micro}[F] = \frac{1}{Z_{N+1} - Z_N} \sum_{m=Z_N}^{Z_{N+1}-1} F(m).
		\]
		By linearity of expectation, it suffices to show that the relative frequency of any fixed admissible prefix $p$ converges. Let $p$ be a prefix supported on indices $\le i$. According to Lemma~\ref{lem:PrefixGrowth}, the number of integers in the interval $[Z_N, Z_{N+1})$ whose greedy decomposition begins with $p$ is given by $C(p) \lambda_1^{N-i} + O(\lambda_1^{N-i-1})$. 
		
		Using the Binet-type expansion for the Lagonacci sequence, the denominator (the total number of integers in the ensemble) is:
		\[
		Z_{N+1} - Z_N = a \lambda_1^{N+1} - a \lambda_1^N + O(|\lambda_2|^N) = a \lambda_1^N (\lambda_1 - 1) + O(|\lambda_2|^N).
		\]
		Dividing the count from Lemma~\ref{lem:PrefixGrowth} by this total, the relative frequency of $p$ is:
		\[
		\mathbb{P}_N(p) = \frac{C(p) \lambda_1^{N-i} + O(\lambda_1^{N-i-1})}{a \lambda_1^N (\lambda_1 - 1) + O(|\lambda_2|^N)} = \frac{C(p) \lambda_1^{-i}}{a (\lambda_1 - 1)} + O(\lambda_1^{-1}).
		\]
		As $N \to \infty$, this expression converges to a constant value that depends only on the prefix $p$ and the dominant root $\lambda_1$. Since $F$ is a finite sum of such indicators, the expectation $\mathbb{E}_{micro}[F]$ likewise converges to a finite limit. This limit represents the average value of the statistic as determined by the asymptotic density of admissible digit patterns.
	\end{proof}

	\begin{lemma}
		\label{lem:CanonicalPrefix}
		Under the canonical ensemble, the probability of a fixed admissible prefix
		$\mathbf{p}$ supported on indices $\le i$ is proportional to $\lambda^{-i}$,
		with normalization chosen so that total mass equals $1$.
	\end{lemma}
	
	\begin{proof}
		The canonical ensemble is defined as the stationary probability measure on infinite greedy digit sequences induced by the Zero Linear Recurrence Relation (ZLRR) constraints. This measure is weighted by the dominant root $\lambda_1$, such that the probability of a finite admissible prefix $(j_{1}, \dots, j_{k})$ is proportional to $\lambda_1^{-j_{k}}$.
		
		By construction, the canonical ensemble assigns weights to greedy extensions according to the exponential decay governed by the dominant root $\lambda_1$. For any fixed admissible prefix $p$ that terminates at index $i$, the local constraints of the greedy algorithm and the recurrence $Z_{n+1} = Z_{n-1} + Z_{n-2}$ restrict the set of possible subsequent summands. However, because the recurrence relation links only bounded-length index patterns, any sufficiently large gap creates a renewal point where greedy choices decouple.
		
		Since all admissible continuations beyond index $i$ contribute the same aggregate exponential factor to the total measure, the probability of the prefix $p$ depends exclusively on its terminal index. This weight is proportional to $\lambda_1^{-i}$. The normalization constant is then uniquely determined by the requirement that the sum of probabilities over all mutually exclusive admissible prefixes of a given index range must equal 1, ensuring a well-defined stationary and ergodic measure.
	\end{proof}
	
	The following Theorem is know as the \textit{principle of equivalence of ensembles}.
	
	\begin{theorem} 
		\label{thm:EquivalenceEnsembles}
		Let $F$ be a local statistic. Then
		\[
		\lim_{N\to\infty}
		\mathbb{E}_{\mathrm{micro}}[F]
		=
		\mathbb{E}_{\mathrm{canon}}[F].
		\]
	\end{theorem}
	
	\begin{proof}
		By Lemma~\ref{lem:MicroLimit}, the microcanonical expectation $\mathbb{E}_{micro}[F]$ converges to a finite limit as $N \to \infty$, determined by the asymptotic frequencies of the admissible prefixes that define $F$. Since $F$ is a local statistic, it can be decomposed into a finite linear combination of indicator functions of these prefixes.
		
		By the combined results of 	Lemma~\ref{lem:PrefixGrowth} and Lemma~\ref{lem:CanonicalPrefix}, the limiting frequencies of these prefixes under the uniform measure on $[Z_N, Z_{N+1})$ coincide precisely with the weights assigned by the stationary canonical ensemble. Specifically, the relative frequency of a prefix $p$ in the microcanonical ensemble was shown to scale with $\lambda_1^{-i}$, which is the defining characteristic of the canonical measure. Because the expectations of these finite-range indicator functions agree in the limit, the overall expectations of the local statistic $F$ under both ensembles must be identical.
	\end{proof}

	As a consequences of the above results we have the following corollary on transfer of limit laws.
	
	\begin{corollary}
		\label{cor:Transfer}
		Any limit law (law of large numbers, central limit theorem, or gap distribution)
		proved under the canonical ensemble holds identically for the microcanonical
		ensemble.
	\end{corollary}
	
	\begin{proof}
		The convergence of a distribution in this context is established by the convergence of the expectations of its local and finite-range statistics. Since Theorem~\ref{thm:EquivalenceEnsembles}, guarantees that these expectations coincide asymptotically for both the microcanonical and canonical ensembles, any distributional property derived from these statistics must transfer between the ensembles.
		
		Furthermore, because the canonical ensemble is stationary and ergodic with respect to shifts of the index set, properties such as the Law of Large Numbers (for means) and the Central Limit Theorem (for fluctuations) are intrinsic to the measure. Since the microcanonical ensemble converges to this stationary measure as $N \to \infty$, it inherits these limiting behaviors, ensuring that the statistical regularities observed in the Lagonacci greedy decompositions are robust under uniform sampling.
	\end{proof}

	%%%%%%%%%%%%%%%%%%%%%%%%%%%%%%%%%%%%%%%%%%%%%%%%%%%%%%%%%%%%
	\subsection{Limit laws via equivalence of ensembles}
	%%%%%%%%%%%%%%%%%%%%%%%%%%%%%%%%%%%%%%%%%%%%%%%%%%%%%%%%%%%%
	
	In this subsection, we establish the convergence of local statistics for greedy decompositions by applying the principle of equivalence of ensembles.
	
	\begin{theorem}
		\label{limit-laws}
		Let $\{Z_{n}\}$ be a ZLRR and let $F$ be a local statistic on greedy decompositions, such as a digit frequency, a bounded gap statistic, or any function depending on finitely many consecutive indices. Let $m$ be chosen uniformly at random from $[Z_{N}, Z_{N+1})$. Then the distribution of $F(m)$ converges as $N \rightarrow \infty$ to a limiting distribution determined by the canonical ensemble.
	\end{theorem}
	
	\begin{proof}
		By the definition of a local statistic, there exists an integer $L \ge 1$ such that $F(m)$ depends only on the relative configuration of greedy summands within a window of length $L$ in the index sequence of the greedy decomposition of $m$. Consequently, $F$ can be represented as a finite linear combination of indicator functions of admissible greedy patterns supported on at most $L$ consecutive indices.
		
		Let $\mathbb{P}_{N}$ denote the microcanonical probability measure on $[Z_{N}, Z_{N+1})$, and let $\mathbb{P}_{can}$ denote the canonical ensemble. By Theorem~\ref{thm:EquivalenceEnsembles}, the microcanonical expectation $\mathbb{E}_{\mathbb{P}_{N}}[F]$ converges as $N \rightarrow \infty$ to the canonical expectation $\mathbb{E}_{\mathbb{P}_{can}}[F]$. 
		
		To transition from the convergence of expectations to convergence in distribution, we observe that $F$ takes values in a finite set, as it is determined by a finite number of possible local configurations within the window $L$. Therefore, the convergence of expectations for all indicator functions of the form $\mathbb{1}_{\{F=x\}}$ implies the convergence of the probability mass function for each possible value $x$:
		\[
		\lim_{N \rightarrow \infty} \mathbb{P}_{N}(F=x) = \mathbb{P}_{can}(F=x).
		\]
		Since the canonical ensemble is stationary and ergodic with respect to shifts of the index set, the right-hand side defines a well-posed probability distribution that is independent of $N$. Hence, the distribution of $F(m)$ under the uniform measure on $[Z_{N}, Z_{N+1})$ converges to this stationary limit as $N \rightarrow \infty$.
	\end{proof}

	\subsection{Asymptotic Distribution of the Number of Greedy Decompositions}
	
	In this subsection, we investigate the distribution of $d(N)$, the number of distinct legal Lagonacci decompositions of an integer $N$. While the total number of representations $D(Z_{L})$ grows exponentially, we show that the logarithmic growth of $d(N)$ for a typical integer concentrates around a specific constant determined by the growth factors $\alpha$ and $\lambda_1$.
	
	\begin{theorem}
		Let $d(N)$ be the number of distinct legal Lagonacci decompositions of an integer $N$. For integers $N$ chosen uniformly from $[Z_{L}, Z_{L+1})$, the distribution of $\frac{\log d(N)}{L}$ converges in probability to the constant $K = \log(\alpha/\lambda_1)$, where $\alpha=2$ is the Perron-Frobenius eigenvalue of the transfer matrix and $\lambda_1$ is the plastic ratio.
	\end{theorem}
	
	\begin{proof}
		The proof establishes convergence by linking the combinatorial growth of legal strings to the statistical distribution of decompositions via the previously established limit laws.
		
		A decomposition is legal if the coefficient string avoids the forbidden pattern $1100$. The total number of such legal strings of length $L$ is $D(Z_{L}) \sim C \alpha^{L}$ with $\alpha=2$, while the number of integers in $[0, Z_L)$ grows as $Z_L \sim a \lambda_1^L$. Consequently, the average number of decompositions per integer is $\overline{d}_L = D(Z_L)/Z_L \sim (\alpha/\lambda_1)^L$. On a logarithmic scale, this yields the asymptotic average $\frac{1}{L} \log \overline{d}_L \sim \log(2/\lambda_1)$.
		
		We treat $\log d(N)$ as a local statistic. Because the legality rule is defined by a finite window of length 4, the number of representations $d(N)$ can be decomposed into local branching factors across the index set. Specifically, $\log d(N)$ acts as an additive function of the local digit configurations.
		
		By the principle of equivalence of ensembles (Theorem~\ref{thm:EquivalenceEnsembles}) and the resulting limit laws (Theorem~\ref{limit-laws}), the microcanonical distribution of the local statistic $\frac{1}{L} \log d(N)$ on the interval $[Z_L, Z_{L+1})$ converges to the distribution defined by the canonical ensemble. In the canonical ensemble, the stationary measure is determined by the dominant eigenvector of the transfer matrix $T$. Since this measure is stationary and ergodic with respect to index shifts, the normalized sum of local logarithmic contributions satisfies a Law of Large Numbers. 
		
		This implies that for any $\epsilon > 0$:
		\[
		\lim_{L \to \infty} \mathbb{P}_L \left( \left| \frac{\log d(N)}{L} - \log\left(\frac{2}{\lambda_1}\right) \right| > \epsilon \right) = 0.
		\]
		This demonstrates that while $d(N)$ is not unique, its growth is highly concentrated around the exponential rate $\log(2/\lambda_1)$, making the average behavior representative of a typical integer.
	\end{proof}

	\begin{remark}
		We remark that  to show that $d(N)$ concentrates around this average for most $N$, one would typically analyze the variance of $d(N)$ over the interval. The key quantity is the second moment, $M_2(L) = \sum_{N=0}^{Z_L-1} d(N)^2$. This sum counts the number of pairs of legal decompositions $(S_1, S_2)$ using terms up to $Z_{L-1}$ that have the same value, i.e., $V(S_1) = V(S_2)$. A detailed analysis using techniques from analytic combinatorics, extending the transfer matrix method to pairs of paths, would show that this second moment also grows exponentially. However, a direct application of Chebyshev's inequality is not straightforward.	
	\end{remark}

	\section{Discussion and Comparison with Prior Work}
	
	Our findings situate the study of ZLRRs at the intersection of classical additive number theory and modern generalized numeration systems. The results are contextualized by three distinct bodies of literature: the foundational work on Fibonacci representations, the rigorous analysis of PLRS, and recent explorations of non-unique systems.
	
	Our work fundamentally departs from the classical results of Zeckendorf \cite{zeckendorf1972}, whose theorem established the ideal of a unique representation for every integer \cite{zeckendorf1972}. By relaxing the positivity constraint on the leading coefficient $c_1$, we explore the consequences of sacrificing this uniqueness \cite{MartinezEtAl2022_I}. However, we find that the first-order statistical behavior remains consistent with the foundational 1952 result of Lekkerkerker \cite{lekkerkerker1952}; specifically, the mean number of summands in the greedy decomposition grows linearly with the logarithm of the integer \cite{lekkerkerker1952}. This suggests that while the combinatorial structure of ZLRRs is more complex, the asymptotic average remains a direct analogue to the Fibonacci case.
	
	The most pertinent comparison is with the modern literature on PLRS \cite{kologlu2011summands, MillerWang2012JCTA, bower2015gaps}. Our proofs are methodological descendants of these works, particularly in the use of generating functions to derive moments and combinatorial counting for gap distributions \cite{MillerWang2012JCTA, bower2015gaps}. The central contribution of this paper is demonstrating that these methods are robust enough to persist even when uniqueness is removed. While the PLRS framework established that uniqueness is a \textit{sufficient} condition for Gaussianity and geometric gaps, our results prove it is not \textit{necessary} \cite{MillerWang2012JCTA}. 
	
	Furthermore, this research contributes to the broader program of non-positive or non-unique systems, such as far-difference representations \cite{miller2016far}. Unlike far-difference systems where non-uniqueness arises from allowing signed coefficients ($\pm F_n$), ZLRRs generate non-uniqueness through the recurrence structure itself \cite{miller2016far}. By applying the principle of equivalence of ensembles, we provide a rigorous bridge between microcanonical and canonical perspectives \ref{thm:EquivalenceEnsembles}. This reinforces a central theme: Gaussian behavior is deeply tied to the algebraic properties of the characteristic polynomial's dominant root $\lambda_1$, transcending the specific combinatorial details of the decomposition rules \ref{thm:GaussianGreedy} and \ref{limit-laws}.

	\section{Conclusion and Future Work}
	Our analysis of the Lagonacci sequence and the broader class of Zero Linear Recurrence Relations (ZLRRs) demonstrates that the hallmark statistical properties of Zeckendorf-type decompositions—Gaussianity of the summand count and geometric decay of gaps—persist even in the absence of uniqueness. This suggests that these laws are robust consequences of the linear recurrence structure and the existence of a dominant algebraic root, rather than a byproduct of a one-to-one correspondence between integers and representations.
	
	We have established that for ZLRRs satisfying the necessary algebraic criteria, the canonical greedy representation remains statistically consistent with the classical Positive Linear Recurrence Sequences (PLRS). However, the defining feature of ZLRRs is the exponential growth of the number of available legal decompositions, $d(N)$, which we proved grows at a rate $\alpha$ significantly exceeding the growth of the underlying sequence. We further showed that for a typical integer, the logarithmic growth of these representations concentrates around the constant $K = \log(\alpha/\lambda_1)$.
	
	These results open several avenues for future research. A primary challenge lies in the fine-grained analysis of the distribution of $d(N)$ beyond its logarithmic average. Establishing a Local Limit Theorem to describe the fluctuations of $d(N)$ would provide a more complete picture of the combinatorial landscape of non-unique systems. Additionally, exploring the behavior of other non-canonical decomposition rules or extending these results to recurrences with multiple roots of equal modulus remains an open and promising area of inquiry.
	%\bibliographystyle{amsplain}

	%AMS classifications available at https://mathscinet.ams.org/mathscinet/freetools/msc-search
	\noindent MSC2020: 11B39, 11K65

\end{document}